\documentclass[11pt]{article}

\usepackage[english]{babel}

\title{Automorphisms of the generalized cluster complex}
\author{Matthieu Josuat-Vergès}
\date{}

\usepackage{stmaryrd}
\usepackage{amsmath,amssymb,amsthm,mathtools}
\usepackage[linktocpage,unicode]{hyperref}
\usepackage{centernot}
\usepackage{tikz}
\usetikzlibrary{positioning}
\usepackage{a4wide}

\usepackage{caption}
\hypersetup{
    pdftoolbar=true,        
    pdfmenubar=true,        
    pdffitwindow=false,     
    pdfstartview={FitH},    
    pdftitle={},
    pdfauthor={},
    pdfsubject={},
    pdfkeywords={},
    pdfnewwindow=true,      
    colorlinks=true,       
    linkcolor=red,          
    citecolor=blue,        
    filecolor=magenta,      
    urlcolor=cyan,           
    unicode=true          
}

\newtheorem{theo}{Theorem}[section]
\newtheorem{coro}[theo]{Corollary}
\newtheorem{lemm}[theo]{Lemma}
\newtheorem{prop}[theo]{Proposition}

\theoremstyle{definition}
\newtheorem{defi}[theo]{Definition}
\newtheorem{rema}[theo]{Remark}

\DeclareMathOperator{\aut}{Aut}
\DeclareMathOperator{\dih}{Dih}
\DeclareMathOperator{\diag}{Diag}
\DeclareMathOperator{\supp}{Supp}
\DeclareMathOperator{\lk}{Link}

\newcommand{\II}{\mathcal{I}} 
\newcommand{\RR}{\mathcal{R}}
\renewcommand{\SS}{\mathcal{S}}
\newcommand{\TT}{\mathcal{T}}
\newcommand{\FF}{\mathcal{F}} 
\newcommand{\GG}{\mathcal{G}} 
\newcommand{\HH}{\mathcal{H}} 
\newcommand{\CC}{\mathcal{C}} 
\newcommand{\DD}{\mathcal{D}} 

\begin{document}

\maketitle

\begin{abstract}
    It is proved that the generalized cluster complex defined by Fomin and Reading has a dihedral symmetry.  Together with diagram symmetries, they generate its automorphism group.  A consequence is a simple explicit formula for the order of this automorphism group. 
\end{abstract}

\section{Introduction}

Fomin and Zelevinsky~\cite{fominzelevinsky} defined the {\it cluster complex} associated to a finite-type cluster algebra: it is a simplicial complex having cluster variables as vertices and clusters as facets.  Furthermore, the construction of Fomin and Zelevinsky involves a pair of maps $\tau_+$ and $\tau_-$ that act as automorphisms of the cluster complex, generating a dihedral group of symmetry.  In type $A_{n+1}$, the cluster complex can be described combinatorially as the complex of {\it dissections} of a regular $(n+2)$-gon (in particular, its facets are triangulations of this polygon), and the dihedral group of symmetries coincides with the apparent geometric symmetry of the $(n+2)$-gon.  The automorphism group of the cluster complex (which is essentially the dihedral group generated by $\tau_+$ and $\tau_-$), and related automorphisms of finite-type cluster algebras, have been investigated in~\cite{assemschiffler,assemschiffler2,changzhu,liuli}.

The {\it generalized cluster complex} $\Gamma^{(m)}$ was introduced by Fomin and Reading~\cite{fominreading}, in the framework of finite-type Coxeter combinatorics.  It depends on an integer parameter $m\geq 1$, in such a way that $m=1$ (in crystallographic types) corresponds to the original definition of Fomin and Zelevinsky.  There is no underlying cluster algebra, but there is a representation theoretic interpretation that lead to many developments such as higher cluster categories.  We refer to~\cite{buan,stumpthomaswilliams,thomas,zhu}.  In particular, Stump, Thomas, and Williams ~\cite{stumpthomaswilliams} give a thorough combinatorial treatment in~\cite{stumpthomaswilliams}. In type $A_{n+1}$, the generalized cluster complex can be described via dissections of a $(mn+2)$-gon (and thus naturally embeds as a subcomplex of the cluster complex of type $A_{mn+1}$).

The original definition of Fomin and Reading makes clear that the complex $\Gamma^{(m)}$ has a cyclic symmetry, via the action of an automorphism $\RR$ (see Section~\ref{sec:review}).  In the case $m=1$, we have $\RR = \tau_+ \tau_-$.  Here, we extend the dihedral symmetry by introducing involutive automorphisms $\SS$ and $\TT$ (see Section~\ref{sec:involutive}) that specialize in $\tau_+$ and $\tau_-$ when $m=1$ and satisfy $\RR^m = \SS\TT$.  This dihedral symmetry group is almost the full automorphism group of $\Gamma^{(m)}$, in the sense that we only need to add extra automorphisms coming from symmetry of the Coxeter graph.  The precise statement is as follows. 

\begin{theo} \label{theo:main}
    Let $W$ be a finite and irreducible Coxeter group, and let $\Gamma^{(m)}$ be the associated generalized cluster complex.  Define two subgroups:
    \begin{itemize}
        \item $\dih \subset \aut(\Gamma^{(m)})$, called the \emph{dihedral subgroup}, is generated by $\RR$ and $\SS$ (or $\RR$ and $\TT$).
        \item $\diag \subset \aut(\Gamma^{(m)})$ is the subgroup of \emph{diagram automorphisms} (See~Section~\ref{sec:diagram}).
    \end{itemize}
    Then we have a semidirect product
    \[
        \aut(\Gamma^{(m)})
        = 
        \dih \rtimes (\diag/\langle     \CC\rangle),
    \]
    where $\CC$ is the {\emph canonical diagram automorphism} (see Section~\ref{sec:canonical}), and
    \begin{equation} \label{eq:ordergroup}
        \big|{\aut(\Gamma^{(m)}) }\big|
            =
        (mh+2)\omega
    \end{equation}
    where $h$ is the Coxeter number of $W$, and $\omega$ is the number of automorphisms of its Coxeter graph (also, $\omega = |{\diag}|$).
\end{theo}

This article is organized as follows:
\begin{itemize}
    \item In Section~\ref{sec:review}, we review the definition and properties of $\Gamma^{(m)}$.
    \item Sections~\ref{sec:reducible} and~\ref{sec:diagram} contain elementary facts about the reducible case and diagram automorphisms, respectively.  The more technical Section~\ref{sec:involutive} introduces the involutive automorphisms $\SS$ and $\TT$.
    \item The proof of Theorem~\ref{theo:main} is contained in Sections~\ref{sec:stabpair}, \ref{sec:stabvert}, \ref{sec:autogroup}.  By a global induction hypothesis, we will assume that the statements in these three sections hold in rank $<n$ where $n$ is the rank of $W$.
\end{itemize}
There are a few final remarks in Section~\ref{sec:final}. 
\section{Review of the generalized cluster complex}

\label{sec:review}

Let $\Phi$ be a finite-type root system, associated to the finite Coxeter group $W$.  Let $\Phi = \Phi_+ \mathrel{\uplus} \Phi_-$ be a decomposition in positive and negative roots.  Let $\Delta\subset \Phi_+$ be the set of simple roots, $S$ the set of simple reflections, $T$ the set of reflections.  For each root $\alpha \in \Phi$, the associated reflection is denoted $t_\alpha \in T$.  For $w\in W$, let $\supp(w) \subset S$ denote the {\it support} of $w$, {\it i.e.}, the set of simple reflections that appear in reduced words for $w$.  The required background on reflection groups is rather small, as our work is mostly based on the combinatorial properties of $\Gamma^{(m)}$ that we review here.

\begin{defi}
    A {\it colored root} is a pair $(\alpha,i) \in \Phi \times \llbracket1,m\rrbracket$. It is denoted $\alpha^i$ and $i$ is called its {\it color}.  It is {\it almost-positive} if $\alpha\in \Phi_+$, or $-\alpha \in \Delta$ and $i=1$.  The set of almost-positive colored roots is denoted $\Phi^{(m)}_{\geq -1}$.
\end{defi}

This set $\Phi^{(m)}_{\geq -1}$ is the vertex set of the generalized cluster complex $\Gamma^{(m)}$.
This is a simplicial complex with many interesting geometric and enumerative features, see~\cite{douvropoulosjosuatverges,fominreading,stumpthomaswilliams,tzanaki}.  Besides the original definition by Fomin and Reading~\cite{fominreading}, a review which is sufficient for our purpose has been given in our previous work~\cite[Section~6]{douvropoulosjosuatverges}.  We give an outline and refer to {\it loc.~cit.} for any further information. 

\subsection{The compatibility relation}

Fomin and Reading originally defined $\Gamma^{(m)}$ as the flag complex associated with a (symmetric) binary relation on $\Phi^{(m)}_{\geq -1}$, called {\it compatibility} of almost-positive colored roots.  This is relevant in the present context, since the automorphism group of a flag complex is clearly the automorphism group of its 1-skeleton.

There exists a black/white decomposition $\Delta = \Delta_\bullet \uplus \Delta_\circ$ where $\Delta_\bullet$, respectively $\Delta_\circ$, contains pairwise orthogonal roots.    The {\it bipartite Coxeter element} $c$ is defined by:
\[
    c_\bullet = \prod_{ \alpha \in \Delta_\bullet } t_\alpha, 
    \quad 
    c_\circ = \prod_{ \alpha \in \Delta_\circ } t_\alpha,
    \quad \text{ and } \quad 
    c = c_\bullet c_\circ.
\]
It is also possible to consider an arbitrary {\it standard Coxeter element} (see~\cite{stumpthomaswilliams}), but the complexes obtained this way are all isomorphic.  So, we stick to the bipartite Coxeter element as in~\cite{fominreading}.

\begin{defi}[\cite{fominreading}]
    The self-bijection $\RR :\Phi^{(m)}_{\geq -1} \to \Phi^{(m)}_{\geq -1} $ is defined by
    \[
        \RR(\alpha^i) :=
        \begin{cases}
            \alpha^{i+1} & \text{ if } \alpha\in\Phi_+ \text{ and } i<m,\\
            (-\alpha)^1 & \text{ if } \alpha \in \Delta_\circ \text{ and } i=m, \text{ or } -\alpha \in \Delta_\bullet \text{ and } i=1,\\
            c(\alpha)^1 & \text{ if } \alpha \in \Phi_+ \backslash \Delta_\circ \text{ and } i=m, \text{ or } -\alpha \in \Delta_\circ \text{ and } i=1. 
        \end{cases}
    \]
\end{defi}

\begin{lemm}[\cite{fominreading}]\label{lemm:equirepartis}
    Let $X\subset \Phi^{(m)}_{\geq -1}$ be an $\RR$-orbit. We have either:
    \begin{itemize}
        \item $\# X = \frac{mh+2}{2}$, and $X$ contains exactly one element of the form $-\rho^1$ (with $\rho\in \Delta$). 
        \item $\# X = mh+2$, and $X$ contains exactly two elements of the form $-\rho^1$ (with $\rho\in \Delta$).  Moreover, these two elements have the form $-\rho^1$ and $w_0(\rho)^1$.
    \end{itemize}
    The order of $\RR$ is given by 
    \begin{equation} \label{eq:orderR}
        \frac{|{\RR}|}{|{w_0}|} = \frac{mh+2}{2}.
    \end{equation}
\end{lemm}

\begin{rema} \label{rema:w0}
    The long element $w_0$ sends $\Delta$ to $-\Delta$, more precisely $\rho \mapsto -w_0(\rho)$ is an automorphism of the Coxeter diagram of $W$ (see also Section~\ref{sec:canonical}).
\end{rema}

\begin{defi}[\cite{fominreading}]
    The {\it compatibility relation} is a symmetric binary relation on $\Phi^{(m)}_{\geq -1}$, uniquely defined by the two conditions:
    \begin{itemize}
        \item  $\mathrel{\|}$ is preserved by $\RR$ ({\it i.e.}, we have $\alpha \mathrel{\|} \beta$ if and only if $\RR(\alpha) \mathrel{\|} \RR(\beta)$),
        \item if $\alpha \in \Delta$, we have $-\alpha^1 \mathrel{\|} \beta^j$ if and only if $t_\alpha \notin \supp(t_\beta)$.
    \end{itemize}
\end{defi}
See~\cite{fominreading} for details on the existence and unicity. 

\begin{defi}[\cite{fominreading}]
    The {\it generalized cluster complex} $\Gamma^{(m)}$ is defined as the flag complex associated with the binary relation $\mathrel{\|}$.  This means that a subset of $\Phi^{(m)}_{\geq-1}$ is a face of $\Gamma^{(m)}$ iff its elements are pairwise compatible. 
\end{defi}

We will use the following rules, which make the compatibility relation more explicit.  Let $\alpha,\beta \in \Delta$, and $1\leq i<j \leq m$.  We have:
\begin{align}
    \label{rule1}
    -\alpha^1 \mathrel{\|} \beta^i &\Longleftrightarrow 
    t_\alpha \notin \supp(t_\beta),\\
    \label{rule2}
    \alpha^i \mathrel{\|} \beta^i &\Longleftrightarrow 
    \langle \alpha | \beta \rangle \geq 0, \text{ and } t_\alpha t_\beta \leq c \text{ or } t_\beta t_\alpha \leq c, \\
    \label{rule3}
    \alpha^i \mathrel{\|} \beta^j
    &\Longleftrightarrow  t_\alpha t_\beta \leq c.
\end{align}

One way to see this is to use the total order on $\Phi^{(m)}_{\geq -1}$ in the next section, see also~\cite[Section~6]{douvropoulosjosuatverges}.  

\subsection{Reflection ordering}

\label{sec:reforder}

An alternative characterization of faces in $\Gamma^{(m)}$ has been given by Tzanaki~\cite{tzanaki}.  It makes a connection with factorization of the Coxeter element.  The idea is to use a total order on $\Phi^{(m)}_{\geq -1}$, akin to Steinberg's indexing of $\Phi$ used by Brady and Watt~\cite{bradywatt}.  Recall that the Steinberg ordering of $\Phi$ is given by indexing roots $(\alpha_i)_{ 1 \leq i \leq nh}$ with the conditions:
\begin{itemize}
    \item $-\Delta_\circ = \{\alpha_i \;:\; 1 \leq i \leq r\}$ (where $r=\#\Delta_\circ$)
    \item $\Delta_\bullet = \{\alpha_i \;:\; r+1 \leq i \leq n\}$
    \item $c(\alpha_i) = \alpha_{i+n}$ (where indices are taken modulo $nh$).
\end{itemize}
Almost-positive roots are the roots $(\alpha_i)_{ 1 \leq i \leq \frac{nh}{2} + n}$, and the indexing gives a total order on $\Phi_{\geq -1}$. There is a decomposition:
\begin{align} \label{eq:phip_decomp}
    \Phi_{\geq -1} = -\Delta_\circ \uplus \Delta_\bullet \uplus c_\bullet(\Delta_\circ) \uplus c(\Delta_\bullet) \uplus \dots \uplus c_\circ(\Delta_\bullet) \uplus  \Delta_\circ \uplus -\Delta_\bullet,
\end{align}
which coarsens the total order in the sense that each block contain consecutive elements and the blocks are written increasingly. Almost-positive roots in $\Phi_{\geq -1}$ are the roots $(\alpha_i)_{1 \leq i \leq \frac{nh}{2}+n}$, and Brady and Watt showed that facets of the cluster complex correspond to decreasing factorizations of the Coxeter element~\cite[Section~8]{bradywatt}. 

Tzanaki's generalization is a total order $\prec$ on $\Phi^{(m)}_{\geq -1}$ given as follows. 
Rather than a total order, it is helpful to think as $\Phi^{(m)}_{\geq -1}$ as a union of $mh+2$ blocks, generalizing~\eqref{eq:phip_decomp}.  Denote $X^i$ for $\{\alpha^i \;:\; \alpha\in X\} \subset \Phi^{(m)}_{\geq-1}$ if $X \subset \Phi$.  We have:
\begin{align} \label{eq:phi_decomp}
    \Phi^{(m)}_{\geq -1}
    = &
    -\Delta_\circ^1 \uplus \Delta_\bullet^m \uplus \dots \uplus \Delta_\circ^{m} \uplus \Delta_\bullet^{m-1} \uplus \dots \uplus \Delta_\circ^{m-1} \uplus\quad \cdots \quad \uplus \\
    & \Delta_\bullet^2 \uplus \dots  \uplus \Delta_\circ^2 \uplus  \Delta_\bullet^1 \uplus \dots \uplus \Delta_\circ^1 \uplus -\Delta_\bullet^1.
\end{align}
The total order is such that the decomposition in~\eqref{eq:phi_decomp} is a coarsening (as above), and each block is ordered via Steinberg's ordering.

\begin{prop}[\cite{tzanaki}] \label{prop:tzanaki}
    Let $f$ be a tuple of $n$ elements of $\Phi^{(m)}_{\geq -1}$, indexed  $f = (\gamma_i^{k_i})_{1\leq i \leq n}$ such that $\gamma_1^{k_1} \succ \dots \succ \gamma_n^{k_n} $.  Then $f$ is a facet of $\Gamma^{(m)}$ iff $c = t_{\gamma_1} \cdots t_{\gamma_n}$.
\end{prop}

\begin{rema} \label{rema:phi_decomp}
    We will see that the decomposition in~\eqref{eq:phi_decomp} is well-behaved with respect to the automorphisms of $\Gamma^{(m)}$.  For example, $\RR$ sends a block of index $i$ to the block of index $i-m$ (where indices are taken modulo $mh+2$).  We will see that $\SS$ fixes $-\Delta_\circ^1$ and reverses the order on the other blocks.  Similarly, $\TT$ fixes $-\Delta_\bullet^1$ and reverses the order on the other blocks.  Note that there might exist nontrivial automorphisms that acts trivially on the decomposition in~\eqref{eq:phi_decomp} (these are the even diagram automorphisms, see Section~\ref{sec:diagram}).
\end{rema}

A direct consequence of the previous proposition is the following:

\begin{lemm}  \label{lemm:csq_compat}
    If $\alpha^k \mathrel{\|} \beta^\ell$, we have $t_\alpha \neq t_\beta$, and $t_\alpha t_\beta \leq c$ or $t_\beta t_\alpha \leq c$. (Here $\leq$ is the \emph{absolute order} on $W$.)  More precisely, $\alpha^k \mathrel{\|} \beta^\ell$ and $\alpha^k \prec \beta^\ell$ imply $t_\beta t_\alpha \leq c$.
\end{lemm}

See~\cite{bradywatt} for more on the absolute order. 

\subsection{The other bipartite Coxeter element}

\label{sec:other}

Note that the definition of $\Gamma^{(m)}$ depends on the choice of $\Delta_\bullet$ and $\Delta_\circ$.  The exchange of $\bullet$ and $\circ$ gives an isomorphic complex and we make this explicit here.

Denote by $ \check{\Gamma}^{(m)}$ the complex defined similar to $\Gamma^{(m)}$ but exchanging the roles of $\Delta_\bullet$ and $\Delta_\circ$.  Similarly, we denote $\check{\|}$ and $\check{\RR}$ the analog of $\|$ and $\RR$. 

\begin{prop}
    The involutive self-map $\iota : \Phi^{(m)}_{\geq -1} \to  \Phi^{(m)}_{\geq -1} $ defined by
    \[
        \iota (\alpha^i) = 
        \begin{cases}
            \alpha^i & \text{ if } \alpha \in -\Delta, \\
            \alpha^{m+1-i} & \text{ if } \alpha \in \Phi_+
        \end{cases}
    \]
    induces an isomorphism $\iota : \Gamma^{(m)} \to \check{\Gamma}^{(m)}$.  Moreover, we have: $\iota \RR \iota = \check{\RR}$. 
\end{prop}

The proof is straightforward and omitted. 

\subsection{Links}

A useful consequence of the definition is the following.  Let $\alpha \in \Delta$, and $W_\alpha$ the maximal standard parabolic subgroup of $W$ obtained by removing $\alpha$ from the Coxeter graph of $W$.  The {\it link} of a face $f\in \Gamma^{(m)}$ is by definition
\[
    \lk(f) = \Big\{ f' \;:\; f\cap f' = \varnothing \text{ and } f \cup f' \in \Gamma^{(m)}\Big\}.
\]
Note that $\lk(f)$ is itself a simplicial complex.

\begin{prop}
    Let $\rho \in \Delta$.  We have $\lk(\{-\rho^1\}) \simeq \Gamma^{(m)}(W_\rho)$. 
\end{prop}

It also follows that the link of any face in $f \in \Gamma^{(m)}$ is isomorphic to $\Gamma^{(m)}(P)$ where $P$ is a standard parabolic subgroup of $W$.  A precise way to give the parabolic subgroup $P$ in terms of $f$ is given in~\cite[Proposition~2.20]{douvropoulosjosuatverges} (this is not needed in the present work).

A concrete consequence of this property of links is the following:

\begin{theo} \label{theo:nonexcep}
    Let $W$ and $W'$ be finite Coxeter groups.  If $\Gamma^{(m)}(W)$ and $\Gamma^{(m')}(W')$ are isomorphic, then $m=m'$ and $W$ and $W'$ are Coxeter-isomorphic.
\end{theo}

\begin{proof}
    We show that $m$ and the Coxeter graph of $W$ can be recovered from $\Gamma^{(m)}(W)$.  The link of a 1-codimensional face in $\Gamma^{(m)}(W)$ is $\Gamma^{(m)}(A_1)$ (since $A_1$ is the unique Coxeter group of rank 1), which consists in $m+1$ isolated vertices.  So $m$ can be recovered from $\Gamma^{(m)}(W)$.  Note also that the rank of $W$ is $1 + \dim(\Gamma^{(m)})$.

    We then proceed by induction on $n$, the rank of $W$.  The case $n=1$ is trivial, and the case $n=2$ is settled by noting that the number of facets of $\Gamma^{(m)}(I_2(k))$ strictly increases with $k$ (this number is known as the Fuß-Catalan number, see~\cite{fominreading} for the exact formula). 

    We need the following: if $n\geq 3$, the Coxeter graph of $W$ is uniquely characterized by the collection of Coxeter graphs of $W_\alpha$ ($\alpha\in \Delta$).  It is straightforward to check this.  As the links of vertices of $\Gamma^{(m)}$ provide this collection (by induction hypothesis), this completes the proof.  
\end{proof}




\section{The reducible case}

\label{sec:reducible}

We examine the situation where $W$ can be decomposed as a product
\[
    W = W_1 \times \dots \times W_r. 
\]
According to \cite{fominreading}, we have 
\begin{equation} \label{eq:join}
        \Gamma^{(m)}(W)
            =
        \Gamma^{(m)}(W_1) \star \dots \star \Gamma^{(m)}(W_r)
\end{equation}
where $\star$ is the join operation on simplicial complexes.  Concretely, this means that each face $f \in \Gamma^{(m)}(W)$ can be written
\begin{equation} \label{eq:f_join}
    f = (f_1,\dots,f_r)
\end{equation}
where $f_i$ is a face of $\Gamma^{(m)}(W_i)$, moreover $\dim(f) = \sum_{i=1}^r \dim(f_i)$.  In particular, at the level of vertex sets we have
\begin{equation}  \label{eq:vert_uplus}
    \Phi^{(m)}_{\geq -1}(W)
        =
    \biguplus_{i=1}^r \Phi^{(m)}_{\geq -1}(W_i).
\end{equation}

\begin{defi}  
    An automorphism $\FF \in \aut(\Gamma^{(m)})$ is {\it monomial} if the partition of the vertex set in~\eqref{eq:vert_uplus} is stabilized: for any $i$, there exists $j$ such that we have $\FF\big( \Phi^{(m)}_{\geq -1}(W_i) \big) = \Phi^{(m)}_{\geq -1}(W_j) $.  
\end{defi}

The terminology is an analogy with monomial matrices.

\begin{prop}  \label{prop:monomial}
    Every element $\FF \in \aut(\Gamma^{(m)})$ is monomial. 
\end{prop}

The idea is to give a combinatorial criterion to characterize when two vertices are in the same block of the partition.

\begin{lemm}
    Let $\alpha^k \in \Phi^{(m)}_{\geq -1}(W_i)$ and $\beta^\ell \in \Phi^{(m)}_{\geq -1}(W_j)$ with  $1\leq i,j \leq r$.  The following conditions are equivalent:
    \begin{itemize}
        \item $i\neq j$,
        \item $\alpha^k \mathrel{\|} \beta^\ell$, and for all $\gamma^p \in \Phi^{(m)}_{\geq -1}$ we have $\alpha^k \mathrel{\|} \gamma^p$ or $\beta^\ell \mathrel{\|} \gamma^p$.
    \end{itemize}
\end{lemm}

\begin{proof}
    First, assume that $i\neq j$.  We have $\alpha^k \mathrel{\|} \beta^\ell$ by~\eqref{eq:join} and by definition of the join.  Let $\gamma^p \in \Phi^{(m)}_{\geq -1}(W_\ell)$.  We have $\alpha^k \mathrel{\|} \gamma^p$ if $i\neq \ell$ and $\beta^\ell \mathrel{\|} \gamma^p$ if $j\neq \ell$.  Since $i\neq j$, at least one of $\alpha^k \mathrel{\|} \gamma^p$ or $\beta^\ell \mathrel{\|} \gamma^p$ thus holds.  Therefore the second condition in the lemma holds. 

    Now, assume $i=j$.  If $\alpha^k \centernot{\|} \beta^\ell$, the second condition in the lemma does not hold.  So, we assume $\alpha^k \mathrel{\|} \beta^\ell$.  It remains to prove the following: there exists $\gamma^p$ such that $\alpha^k \centernot{\|} \gamma^p$ and $\beta^\ell \centernot{\|} \gamma^p$.  By using using the map $\RR$ on the irreducible factor $W_i$ and the invariance of compatibility, we can assume that $\alpha \in -\Delta$ (and $k=1$).  
    The construction of $\gamma^p$ is as follows.
    \begin{itemize}
        \item If $\beta\in -\Delta$ (and $\ell=1$), let $\gamma \in \Phi_+$ be such that $t_\alpha \in \supp(t_\gamma)$ and $t_\beta \in \supp(t_\gamma)$ (it exists because $t_\alpha$ and $t_\beta$ are both in $W_i$ which is irreducible, and a finite irreducible Coxeter group has reflections with full support).  The color $p$ can be anything. 
        \item Otherwise, let $P \subset W_i$ be a standard parabolic subgroup such that $P$ is irreducible, $t_\alpha \in  P$, $P \mathrel{\cap} \supp(t_\beta) = \varnothing$, and $P$ contains a neighbor of $\supp(t_\beta)$ in the Coxeter graph. (It exists because $t_\alpha$ and $t_\beta$ are both in $W_i$ which is irreducible, and these properties easily translate into properties of the corresponding subset of $S$ or $\Delta$.  We used the fact that $t_\alpha \notin \supp(t_\beta)$, which holds since $\alpha^k \mathrel{\|} \beta^\ell$).  Let $\gamma \in \Delta$ such that $\supp( t_\gamma ) = P\cap S$, and $p$ is arbitrary.  Now, we have: 
        \begin{itemize}
            \item $\alpha^k \centernot{\|} \gamma^\ell$ holds via~\eqref{rule1}, since $t_\alpha \in P \mathrel{\cap} S = \supp(t_\gamma)$.
            \item $\beta^\ell \centernot{\|} \gamma^\ell$ holds via~\eqref{rule3}, since $\langle \beta | \gamma \rangle <0$.  This is obtained as follows.  Write
            \[
                \beta = \sum_{\rho \in \Delta,\; t_\rho \in \supp(t_\beta)} x_\rho\cdot \rho,
                \qquad 
                \gamma = \sum_{\rho \in \Delta,\; t_\rho \in \supp(t_\gamma)} y_\rho\cdot \rho
            \]
            with $x_\rho>0$ and $y_\rho>0$.  By expanding $\langle \beta | \gamma \rangle$, the unique nonzero term is $x_\sigma y_\tau \langle \sigma |\tau \rangle$ where $\sigma$ and $\tau$ are neighbors in the Coxeter graph, so that $\langle \sigma |\tau \rangle <0$.
        \end{itemize}
    \end{itemize}
    In both cases, we have found $\gamma^p$ with the desired properties.  This shows that if $i=j$, the second condition in the lemma doesn't hold.
\end{proof}

\begin{proof}[Proof of Proposition~\ref{prop:monomial}]
    Let $\alpha \in \Phi^{(m)}_{\geq -1}(W_i)$ and $\beta \in \Phi^{(m)}_{\geq -1}(W_j)$ with  $1\leq i,j \leq k$, and $\FF\in \aut(\Gamma^{(m)})$.   Because the second condition in the previous lemma is invariant under $\FF$, $\alpha$ and $\beta$ are in the same block of the vertex set partition if and only if $\FF(\alpha)$ and $\FF(\beta)$ have the same property.  So, $\FF$ stabilizes the vertex set partition in~\eqref{eq:vert_uplus}, which means that it is a monomial automorphism. 
\end{proof}

Let us describe one particular application of the previous proposition. 

\begin{rema} \label{rema:linkfactors}
    Assume now that $W$ is irreducible, let $\alpha \in \Delta$, and let $W' = W_{\alpha}$ be the maximal standard parabolic subgroup of $W$ obtained by removing $\alpha$ in the Coxeter graph of $W$.  Let $\FF \in \aut(\Gamma^{(m)}(W))$.  If $\FF(-\alpha^1) = -\alpha^1$, the restriction of $\FF$ to the link of $-\alpha^1$ is also an automorphism that we denote $\FF' \in \aut( \Gamma^{(m)}(W') )$.  Because every automorphism of $\Gamma^{(m)}(W')$ is monomial, there is a type-preserving permutation of the irreducible factors of $W'$ underlying $\FF$.  This will be helpful in understanding the stabilizer of $-\alpha^1$ in $\aut(\Gamma^{(m)}$).
\end{rema}

\section{Diagram automorphism}

\label{sec:diagram}

Let $\DD : \Delta \to \Delta$ be an automorphism of the Coxeter graph of $W$.  There is an induced automorphism $W\to W$, and an induced self-bijection $\Phi \to \Phi$.   We keep the notation $\DD$ for these induced maps. 

\subsection{Even diagram automorphisms}

First, we assume that $\Delta_\bullet $ and $\Delta_\circ$ are preserved by $\DD$. A non-trivial such automorphism exists in types $A_n$  with $n$ odd, $E_6$, $D_n$ with $n\geq 4$ (and it is unique except for type $D_4$ where the group of even diagram automorphism is the symmetric group $\mathfrak{S}_3$).

We extend $\DD$ to $\Phi^{(m)}_{\geq -1}$ by $\DD(\alpha^i) = \DD(\alpha)^i$.  It is straighforward to check that the map $\DD$ preserves the compatibility relation, so that $\DD \in \aut(\Gamma^{(m)})$.  It is called an {\it even diagram automorphism}.  Moreover, we have $\DD\RR = \RR\DD$.

\subsection{Odd diagram automorphisms}

Now, assume that $\Delta_\bullet $ and $\Delta_\circ$ are exchanged by $\DD$.  A non-trivial such automorphism exists in types $A_n$ with $n$ even, $F_4$ and $I_2(k)$ (and it is unique).

We extend $\DD$ to $\Phi^{(m)}_{\geq -1}$ by 
\[  
    \DD(\alpha^i) =
    \begin{cases}
        \DD(\alpha)^i & \text{ if } \alpha \in -\Delta, \\
        \DD(\alpha)^{m+1-i} & \text{ if } \alpha \in \Phi_+.
    \end{cases}
\]
It is straightforward to check that this gives an automorphism $\DD \in \aut(\Gamma^{(m)})$, which is called an {\it odd diagram automorphism}.  Moreover, we have $\DD \RR = \RR^{-1} \DD$.

\subsection{The canonical diagram automorphism}

\label{sec:canonical}

Let us recall standard facts of Coxeter theory.  Conjugation by the longest element $w_0$ acts on simple reflections as a symmetry of the Coxeter diagram (see also Remark~\ref{rema:w0}).  It is the identity when all exponents are even, and it is the unique nontrivial symmetry of order 2 in other cases ($A_n$, $D_n$ with $n$ odd, $E_6$, $I_2(2k+1)$).    If $h$ is even, we have $c^{h/2} = w_0$.  More precisely, for $\alpha \in \Delta$ we have $-c^{h/2}(\alpha) \in \Delta$ and it is the image of $\alpha$ under the symmetry of the Coxeter diagram. 


\begin{defi}
    The {\it canonical diagram automorphism} $\CC \in \diag$ is the element of $\diag$ associated to the map $\Delta\to \Delta$, $\alpha \mapsto -w_0 (\alpha)$.
\end{defi}

\begin{lemm} \label{lemm:Rpow}
    If $h$ is even, we have $\RR^{(mh+2)/2} = \CC$.
\end{lemm}

\begin{proof}
    By the definition of $\RR$, we get 
    \begin{align} \label{eq:Rm}
        \RR^m(\alpha^i) 
        =
        \begin{cases}
            c(\alpha)^i & \text{ if } \alpha \in \Phi_+\backslash \Delta_\bullet, \\
            -c(\alpha)^{i-1} & \text{ if } \alpha \in \Delta_\circ \text{ and } i>1,\\
            (-\alpha)^1 & \text{ if } \alpha\in \Delta_\circ \text{ and } i=1,\\
            c(\alpha)^m & \text{ if } \alpha \in -\Delta_\circ \text{ and } i=1, \\
            (-\alpha)^m & \text{ if } \alpha \in -\Delta_\bullet \text{ and }i=1.            
        \end{cases}
    \end{align}
    From this, it is straightforward to compute $\RR^{mh/2}$ in the different cases, and show that $\RR^{mh/2+1}$ is the map $\alpha^i \mapsto -c^{h/2}(\alpha)^i$.
\end{proof}

\begin{prop} \label{prop:c_in_dih}
    We have $\CC \in \dih$.
\end{prop}

\begin{proof}
    Via the previous lemma, it remains only to consider the case where $h$ is odd, {\it i.e.}, $W$ has type $A_{2n}$.  Using the combinatorial model of polygon dissections, the result is clear.
\end{proof}

Recall from the introduction that $\diag \subset \aut(\Gamma^{(m)})$ is defined as the subgroup of diagram automorphisms.  To complete this section, note that Lemma~\ref{lemm:stab_negface} below characterizes elements of $\diag$ as the automorphisms of $\Gamma^{(m)}$ that stabilize $-\Delta^1$ (setwise).

\section{Involutive automorphisms}

\label{sec:involutive}

We define the two involutive automorphisms $\SS$ and $\TT$, as outlined in the introduction.   


\begin{defi}
    The self-map $\SS$ on $\Phi^{(m)}_{\geq -1}$ is defined by:
    \begin{align} \label{eq:defS}
        \SS( \alpha^i )
        =
        \begin{cases}
            \alpha^i & \text{ if } \alpha \in -\Delta_\circ \text{ and } i=1 \text{ (i)}, \\
            (-\alpha)^m & \text{ if } \alpha \in -\Delta_\bullet \text{ and } i=1 \text{ (ii)}, \\
            (-\alpha)^1 & \text{ if } \alpha\in \Delta_\bullet \text{ and } i=m \text{ (ii')},\\
            \alpha^{m-i} & \text{ if } \alpha\in \Delta_\bullet \text{ and } 1\leq i \leq m-1 \text{ (iii)},\\
            c_\bullet(\alpha)^{m+1-i} & \text{ if} \alpha \in \Phi_+ \backslash \Delta_\bullet \text{ (iv)}.
        \end{cases}
    \end{align}
\end{defi}

To gain some insight about this definition, we let the reader check how $\SS$ acts on the vertex partition in~\eqref{eq:phi_decomp} as explained in Remark~\ref{rema:phi_decomp}.  Another helpful verification is the following: in the combinatorial model of polygon dissections (in type $A_n$), this map acts geometrically as a reflection of the polygon and its diagonals.

\begin{lemm}
    We have $\SS^2 = \II$.  
\end{lemm}

\begin{proof}
    This is straightforward from the different cases in the definition and from $c_\bullet$ being an involution on $\Phi_+ \backslash \Delta_\bullet$.
\end{proof}

\begin{prop} \label{lemm:SSaut}
    The map $\SS$ induces an automorphism of $\Gamma^{(m)}$.  
\end{prop}

\begin{proof}
We show that the map $\SS$ sends a facet of $\Gamma^{(m)}$ to another facet, using the description in terms of reduced factorization of the Coxeter element (Section~\ref{sec:reforder}). 

Let $ f = (\alpha_i^{k_i})_{1\leq i \leq n }$ be a facet of $\Gamma^{(m)}$, indexed as in Proposition~\ref{prop:tzanaki}, so that $c = t_{\alpha_1} \cdots t_{\alpha_n}$. The {\it canonical factorization} of $c$ associated with a facet of $\Gamma^{(m)}$ is:
\begin{equation}
    c = w_\bullet w_1 \cdots w_m w_\circ,
\end{equation}
defined as a coarsening of $t_{\alpha_1} \cdots t_{\alpha_n}$ where $w_\bullet$ (respectively, $w_j$, $w_\circ$) contains the factors $t_{\alpha_i}$ such that $\alpha_i$ is in $-\Delta_\bullet$ (respectively, $\Phi_+^j$, $-\Delta_\circ$).  We further refine this by writing
\[
    w_j = w_{j,2} w_{j,1}
\]
where $ w_{j,1}$ contains all factors $t_{\alpha_i}$ with $\alpha_i \in \Delta_\bullet$.  

It is helpful to first assume that $f' := \SS(f)$ is also a facet of $\Gamma^{(m)}$, and compute what should be the associated canonical factorization.  Denote it:
\[
    c = w'_\bullet w'_1 \cdots w'_m w'_\circ,
\]
and again we write
\[
    w'_i = w'_{i,2} w'_{i,1}
\]
where $ w'_{i,1}$ contains all factors $t_\alpha$ with $\alpha \in \Delta_\bullet$.  Now, the definition of $\SS$ gives necessary relations between all the factors we just defined:
\begin{align*}
    w'_\circ    &= w_\circ \text{ (from \eqref{eq:defS}, case (i))},\\
    w'_{i,1}    &= w_{m-i,1} \; \text{where } 1\leq i \leq m-1 \text{ (from \eqref{eq:defS}, case (iii))},\\
    w'_\bullet  &= w_{m,1} \text{ (from \eqref{eq:defS}, case (ii))},\\
    w'_{m,1}    &= w_\bullet \text{ (from \eqref{eq:defS}, case (ii'))},\\
    w'_{i,2}    &= c_\bullet  w_{m+1-i,2}^{-1} c_\bullet \; \text{where } 1\leq i \leq m \text{ (from \eqref{eq:defS}, case (iv))}. 
\end{align*}
About the latter equality, let us explain why we need to take the inverse.  Let $\alpha_1^i,\dots,\alpha_k^i \in \Phi^{(m)}_{\geq -1}$ be the vertices of $f$ that contributes to the factors of $w_{i,2}$, ordered so that $w_{i,2} = t_{\alpha_1} \cdots t_{\alpha_k}$ ({\it i.e.}, decreasingly with respect to $\prec$).  Let $\alpha'_i = c_\bullet(\alpha_i)$, so that $t_{\alpha'_i} = c_\bullet t_{\alpha_i} c_\bullet$.  Note that $w_{i,2} \leq c$ implies $c_\bullet w_{i,2} c_\bullet \leq c_\bullet c c_\bullet = c^{-1} $, so that $c_\bullet w_{i,2}^{-1} c_\bullet \leq c$ (it is well-known and easy to see from the definition that the absolute order is invariant under conjugation).  The way to order the elements $t_{\alpha'_i}$ and get an element below $c$ in the absolute order is thus $c_\bullet w_{i,2}^{-1} c_\bullet = t_{\alpha'_k} \cdots t_{\alpha'_1}$.

Now, let us check that these are indeed the factors of a factorization of $c$:
\begin{align*}
     w'_\bullet w'_1 \cdots w'_m w'_\circ 
     &= w'_\bullet w'_{1,2} w'_{1,1} \cdots w'_{m,2} w'_{m,1} w'_\circ \\
     &= w_{m,1} \cdot c_\bullet w_{m,2}^{-1} c_\bullet \cdot w_{m-1,1} \cdot c_\bullet w_{m-1,2}^{-1} c_\bullet \cdot w_{m-2,1} \cdots  c_\bullet w_{1,2}^{-1} c_\bullet \cdot w_\bullet w_\circ.
\end{align*}
Because $c_\bullet$ commutes with $w_{i,1}$, this gives:
\begin{align*}
     w'_\bullet w'_1 \cdots w'_m w'_\circ 
     &= w_{m,1} \cdot c_\bullet \cdot w_{m,2}^{-1} w_{m-1,1} w_{m-1,2}^{-1} w_{m-2,1} \cdots w_{1,2}^{-1} c_\bullet \cdot  w_\bullet w_\circ \\
     &= c_\bullet (w_1 \cdots w_m)^{-1} \cdot c_\bullet w_\bullet w_\circ \\
     &= c_\bullet (w_\circ c^{-1} w_\bullet) \cdot c_\bullet w_\bullet w_\circ
     = c_\bullet w_\circ c_\circ c_\bullet w_\bullet c_\bullet w_\bullet w_\circ
     = c_\bullet w_\circ c_\circ w_\circ
     =c.
\end{align*}


The previous computations prove the proposition.  Indeed, if $f = \{\alpha_1^{i_1},\dots, \alpha_n^{i_n}\}$ and $f' = \{ \SS(\alpha_1^{i_1}),\dots, \SS(\alpha_n^{i_n}) \}$, the factorization $c = w'_\bullet w'_1 \cdots w'_m w'_\circ$ can be refined as a reflection factorization that proves $f' \in \Gamma^{(m)}$ via Proposition~\ref{prop:tzanaki}.

    An alternative proof would be check that $\SS$ preserves the compatibility relation $\|$. This is straightforward, although a bit long because of the various cases to consider. 
\end{proof}

\begin{lemm}
    We have $\SS\RR\SS = \RR^{-1}$.
\end{lemm}

\begin{proof}
    We make the composition $\mathcal{SR}$ explicit. First,
    \[
        \mathcal{SR}(\alpha^i) = 
        \begin{cases}
            \SS(\alpha^{i+1}) = \alpha^{m-i-1}  & \text{ if } \alpha \in \Delta_\bullet \text{ and } 1 \leq i \leq m-2, \\
            \SS(\alpha^{i+1}) = c_\bullet(\alpha)^{m-i}  & \text{ if } \alpha \in \Phi_+ \backslash \Delta_\bullet \text{ and } 1 \leq i \leq m-1, 
        \end{cases}
    \]
    so $\mathcal{SR}$ is an involution on the elements considered in each case.  Second,
    \[
        \mathcal{SR}(\alpha^i) = 
        \begin{cases}
            \SS( (\alpha)^1) = (-\alpha)^1  & \text{ if } \alpha \in \Delta_\circ \text{ and } i = m, \\
            \SS( c(\alpha)^1 ) = \SS( c_\bullet( -\alpha)^1 ) = (-\alpha)^m  & \text{ if } \alpha \in - \Delta_\circ \text{ and } i = 1,
        \end{cases}
    \]
    and
    \[
        \mathcal{SR}(\alpha^i) = 
        \begin{cases}
            \SS( (-\alpha)^1) = (-\alpha)^{m-1}  & \text{ if } \alpha \in -\Delta_\bullet \text{ and } i = 1, \\
            \SS( \alpha^m ) = (-\alpha)^1  & \text{ if } \alpha \in \Delta_\bullet \text{ and } i = m-1,
        \end{cases}
    \]
    so $\SS \RR$ is also an involution on these elements.  Eventually, we have:
    \[
        \SS \RR (\alpha^i) = \SS( c(\alpha)^1) = c_\circ (\alpha)^m \quad \text{ if } \alpha \in \Phi_+ \backslash \Delta_\circ \text{ and } i=m. 
    \]
    (To check this, distinguish the cases $\alpha\in \Delta_\bullet$, $\alpha\in c_\circ(\Delta_\bullet)$, and $\alpha \in \Phi_+ \backslash( \Delta_\circ \cup \Delta_\bullet \cup c_\circ(\Delta_\bullet) )$).  It is now clear that $\mathcal{SR}$ is an involution.
\end{proof}

The previous lemma means that $\langle\RR,\SS\rangle \subset \aut(\Gamma^{(m)})$ is a dihedral subgroup.  Recall from the introduction that it is denoted $\dih$.

\begin{defi}
    The self-map $\TT$ on $\Phi^{(m)}_{\geq -1}$ is defined by:
    \begin{align}
        \TT( \alpha^i )
        =
        \begin{cases}
            \alpha^i & \text{ if } \alpha \in -\Delta_\bullet \text{ and } i=1 \text{ (i)}, \\
            -\alpha^{1} & \text{ if } \alpha\in \pm \Delta_\circ \text{ and } i=1 \text{ (ii)},\\
            \alpha^{m+2-i} & \text{ if } \alpha\in \Delta_\circ \text{ and } 2 \leq i \leq m \text{ (iii)},\\
            c_\circ(\alpha)^{m+1-i} & \text{ if } \alpha \in \Phi_+ \backslash \Delta_\circ  \text{ (iv)}. 
        \end{cases}
    \end{align}
\end{defi}

\begin{prop}
    We have $\TT \in \aut(\Gamma^{(m)})$, moreover $\TT =  \iota  \SS  \iota $ ({\it i.e.}, $\check{\SS} = \TT$ and $\check{\TT} = \SS$).  
\end{prop}

The proof is straightforward, by black/white symmetry.

\begin{rema}
    When $m=1$, these maps $\SS$ and $\TT$ are those considered by Fomin and Zelevinsky in~\cite{fominzelevinsky}.  One can check that $\SS \TT = \RR^m$ (the map $\RR^m$ is explicit in~\eqref{eq:Rm}).  The two maps $\SS$ and $\TT$ are thus generators of $\dih$ if $m$ is relatively prime with the order of $\RR$.  This happens in the following two situations:
    \begin{itemize}
        \item the order of $\RR$ is $mh+2$ and $m$ is relatively prime to $mh+2$ (i.e., $m$ is odd).
        \item the order of $\RR$ is $\frac{mh+2}{2}$ ($h$ is even in this case, so that $m\frac{h}{2}+1$ is relatively prime to $m$).
    \end{itemize}
    Because this is not exhaustive, it is not possible to build the dihedral group of symmetries with $\SS$ and $\TT$ as generators, although they look like a natural set of generators. 
\end{rema}

\section{Stabilizer of a pair of vertices}

\label{sec:stabpair}

Here and in the following sections, we take $\alpha,\beta \in \Delta$ such that $\beta$ is the unique neighbor of some $\alpha$ in the Coxeter graph of $W$ ($\alpha$ is a {\it leaf}).  By symmetry (more precisely, using the isomorphism between $\Gamma^{(m)}$ and $\check{\Gamma}^{(m)}$ from Section~\ref{sec:other}), we can assume $\alpha \in \Delta_\bullet$ and $\beta \in \Delta_\circ$ in Proposition~\ref{prop:stab2vert} and the other case follows.  Let $W_\beta$ be the maximal standard parabolic subgroup obtained by removing $\beta$ in the Coxeter graph of $W$, and similarly for $\alpha$.   


\begin{prop} \label{prop:stab2vert}
    The (pointwise) stabilizer of $ \{-\alpha^1, -\beta^1 \}$ in $\aut(\Gamma^{(m)})$ is generated by diagram automorphisms.  
\end{prop}

\begin{lemm} \label{lemm:stab_negface}
    Assume that $W$ is irreducible, and its rank is at least $2$.  If $\FF \in \aut(\Gamma^{(m)})$ is such that $\FF(-\rho^1) = -\rho^1$ for all $\rho\in \Delta$, then $\FF = \II$. 
\end{lemm}

\begin{proof}
    Consider the restriction of $\FF$ to the link of a vertex $-\rho^1$, for $\rho \in \Delta$.  Using an induction hypothesis, we can apply Proposition~\ref{prop:stabvertex} to the maximal standard parabolic subgroup $W_\rho$, and find that the restriction of $\FF$ on $\Gamma^{(m)}(W_\rho)$ is the identity.  In particular, $\FF(\rho^i)$ for $\rho\in\Delta$ and $1\leq i \leq m$ (since $\rho^i$ is in the link of $-\tau^1$ if $\tau\in\Delta$ and $\tau\neq\rho$). 
    
    To establish an induction showing that $\FF=\II$, we consider the block decomposition in~\eqref{eq:phi_decomp}. 
    \begin{itemize}
        \item We have shown that if the two neighbor blocks $-\Delta_\circ^1$ and $-\Delta_\bullet^1$ contain only fixed points, this is also the case (in particular) for the next blocks $\Delta_\circ^1$ and $\Delta_\bullet^m$.  
        \item The automorphisms $\RR$, $\SS$, and $\TT$ can be used to send any pair of neighbor blocks in~\eqref{eq:phi_decomp} to another one. Therefore, the previous property holds for any pair of consecutive blocks, not just $-\Delta_\circ^1$ and $-\Delta_\bullet^1$.
    \end{itemize}
    We conclude that every element of $\Phi^{(m)}_{\geq -1}$ is fixed by $\FF$.
\end{proof}

In particular, the previous lemma settles the case of rank 2 in Proposition~\ref{prop:stab2vert}.  We thus assume that the rank of $W$ is at least $3$ in the rest of this section.

\begin{lemm} \label{lemm:rank3}
    Suppose that the rank of $W$ is $3$.  Write $\Delta_\bullet = \{ \alpha, \gamma\}$ and $\Delta_\circ = \{\beta\}$.  There exists no automorphism $\FF \in \aut(\Gamma^{(m)})$ such that $\FF(-\alpha^1) = -\alpha^1$ and $\FF(-\gamma^1) = \gamma^m$.
\end{lemm}

\begin{proof}
    By the finite-type classification, we have $t_\alpha t_\beta t_\alpha = t_\beta t_\alpha t_\beta $ or $t_\gamma t_\beta t_\gamma = t_\beta t_\gamma t_\beta $.  The statement is symmetric in $\alpha$ and $\gamma$, because $\FF(-\alpha^1) = -\alpha^1$ and $\FF(-\gamma^1) = \gamma^m$ is equivalent to $\SS\FF(-\alpha^1) = \alpha^m$ and $\SS\FF(-\gamma^1) = -\gamma^1$.  So, we assume $t_\gamma t_\beta t_\gamma = t_\beta t_\gamma t_\beta $.
    
    Consider the 1-dimensional face $f = \{-\alpha^1, -\gamma^1 \}$.  We have:
    \begin{equation} \label{eq:propf}
        \forall \rho^\ell \text{ vertex of } \lk(f),\; \lk(\{ \rho^\ell \}) \simeq \Gamma^{(m)}(A_1\times A_1). 
    \end{equation}
    Indeed, via~\eqref{rule1}, $\lk(f)$ is the $0$-dimensional complex with vertices $\{-\beta^1, \beta^1 , \dots , \beta^m \}$.  These vertices are in the $\RR$-orbit of $-\beta^1$.  So their links are all isomorphic to the link of $-\beta^1$, which is $\Gamma^{(m)}(W_\beta) \simeq \Gamma^{(m)}(A_1\times A_1$).

    Now, let $f' = \{ -\alpha^1, \gamma^m \}$.  Since the property in~\eqref{eq:propf} is invariant under automorphisms, it remains only to show that this property doesn't hold with $f'$ in place of $f$ to conclude that no automorphism sends $f$ to $f'$.  To do this, let $\rho = t_\gamma(\beta)$.  We check:
    \begin{itemize}
        \item $f' \cup\{ \rho^m \} \in \Gamma^{(m)}$.  The relation $-\alpha^1 \mathrel{\|} \rho^m$ is clear via~\eqref{rule1}.  The relation $\gamma^m \mathrel{\|} \rho^m$ holds via~\eqref{rule3} because: i) $t_\rho = t_\gamma t_\beta t_\gamma$ so that $t_\rho t_\gamma = t_\gamma t_\beta \leq c $,
        and ii) $\langle \rho | \gamma \rangle = \langle t_\gamma(\beta) | \gamma \rangle = -  \langle \beta | \gamma \rangle >0 $.
        \item $\lk(\{\rho^m\} ) \not \simeq \Gamma^{(m)}( A_1\times A_1 )$.  We have $t_\rho = t_\gamma t_\beta t_\gamma = t_\beta t_\gamma t_\beta$ (by the assumption at the beginning of this proof).  Since $c_\circ = t_\beta$, it follows that $c_\circ(\rho) = \gamma$, and $\TT( \rho^m ) = \gamma^1 = \RR( -\gamma^1) $.  So, $\{ \rho^m \}$ and $\{ -\gamma^1 \}$ have isomorphic links.  Since $W$ is irreducible, $W_\gamma \not \simeq A_1 \times A_1$, and $\lk(\{\rho^m\} ) \simeq \Gamma^{(m)}( W_{\gamma} ) \not \simeq \Gamma^{(m)}(A_1\times A_1) $ by Theorem~\ref{theo:nonexcep}.
    \end{itemize}
    So,~\eqref{eq:propf} doesn't hold with $f'$ in place of $f$. 
\end{proof}

\begin{proof}[Proof of Proposition~\ref{prop:stab2vert}]
    Let $\FF \in \aut(\Gamma^{(m)})$ with $\FF(-\alpha^1) = -\alpha^1 $ and $\FF(-\beta^1) = -\beta^1$.   Since $\FF(-\alpha^1) = -\alpha^1$, the restriction of $\FF$ to the link of $-\alpha^1$ gives an element $\FF' \in \aut( \Gamma^{(m)}(W_\alpha) )$.  Since $\{-\alpha^1, -\beta^1\} \in \Gamma^{(m)}$, $-\beta^1$ is in the link of $-\alpha^1$ and naturally identifies to a vertex $-\beta^1 \in \Gamma^{(m)}(W_\alpha)$ which is fixed by $\FF'$.  Using an induction hypothesis, we can apply Proposition~\ref{prop:stabvertex2}.  
    
    Proposition~\ref{prop:stabvertex2} gives $\FF' = \SS' \DD'$ or $\FF' = \DD'$, where $\DD'$ is a diagram automorphism of $\Gamma^{(m)}(W_\alpha)$ fixing $-\beta^1$ (and $\mathcal{S}'$ is the restriction of $\mathcal{S}$ in $\aut( \Gamma^{(m)}(W_\alpha) )$).  Note that $\DD'$ is the restriction of a diagram automorphism $\DD$ of $\Gamma^{(m)}(W)$ fixing $-\alpha^1$ and $-\beta^1$ (this is easily checked at the level of Coxeter graphs).  Our goal is to show that $\FF = \DD$.  Let $\GG := \FF \DD^{-1}$.
    
    By construction $\GG$ fixes $-\alpha^1$ and $-\beta^1$, moreover its restriction on $\Gamma^{(m)}(W_\alpha)$, denoted $\GG'$, is $\SS'$ or the identity.  By way of contradiction, assume that $\GG' = \SS'$.  Let $W_I$ be the standard parabolic subgroup of $W$ with simple roots $\beta$ and its neighbors in the Coxeter graph.  If $\gamma \in \Delta$ is at distance 2 from $\beta$, the vertex $-\gamma^1$ is fixed by $\SS'$ (because $\gamma\in \Delta_\circ$, just like $\beta$), and consequently it is also fixed by $\GG$.  It follows that the restriction of $\GG$ gives an automorphism $\HH \in \aut( \Gamma^{(m)}(W_I) )$.  The situation can be summarized as follows.
    \begin{itemize}
        \item By definition of $W_I$, its Coxeter graph is a star centered at $\beta$.  By the finite-type classification, either it has rank 3, or it is of type $D_4$. 
        \item The vertices $-\alpha^1$ and $-\beta^1$ are fixed by $\HH \in \aut(\Gamma^{(m)}(W_I))$.  The other vertices $\gamma$ and possibly $\delta$ (if $W_I$ has type $D_4$) are such that $\HH(-\gamma^1) = \gamma^m$  and $\HH(-\delta^1) = \delta^m$.
    \end{itemize}
    If the rank of $W_I$ is $3$, Lemma~\ref{lemm:rank3} gives a contradiction, so that such $\HH$ doesn't exist.  In the case where $W_I$ has type $D_4$, consider the composition $\SS \HH$: $-\beta^1$, $-\gamma^1$, $-\delta^1$ are fixed, and the image of $-\alpha^1$ is $\alpha^m$.  By considering the link of $-\delta^1$, we see that the nonexistence of such $\HH$ in type $D_4$ follows from the nonexistence in type $A_3$. 

    At this point, we have proved that $\GG'$ is the identity.  This proves that $\GG$ fixes all the vertices $-\alpha^1$ for $\alpha\in \Delta$.  By Lemma~\ref{lemm:stab_negface}, it follows that $\GG=\II$, so $\FF=\DD$. (Note that the assumption on the rank in Lemma~\ref{lemm:stab_negface} holds, because we assumed that the rank of $W$ is at least 3.)
\end{proof}

\section{Stabilizer of vertices}

\label{sec:stabvert}

Recall that $\alpha \in \Delta_\bullet$ and $\beta \in \Delta_\circ$ are such that $\beta$ is the unique neighbor of $\alpha$ in the Coxeter graph of $W$.  The goal of this section is to describe the stabilizer of the vertex $-\beta^1$  in $\aut(\Gamma^{(m)})$.  

\begin{prop} \label{prop:stabvertex}
    Let $\FF \in \aut(\Gamma^{(m)})$ be such that $\FF(-\beta^1) = -\beta^1$.  Then we have either $\FF = \DD$ or $\FF = \SS \DD$, where $\DD$ is an even diagram automorphism such that $\DD(-\beta^1) = -\beta^1$.
\end{prop}


\begin{lemm} \label{lemm:squares}
    Let $\FF \in \aut(\Gamma^{(m)})$ be such that $\FF(-\beta^1) = -\beta^1$.  If $\FF$ stabilizes the vertex set $\{ -\alpha^1, \alpha^1, \dots, \alpha^m\}$,  then it also stabilizes the vertex set $\{-\alpha^1,\alpha^m\}$.
\end{lemm}

\begin{proof}
    In this proof, a tuple of 4 elements of $\Phi^{(m)}_{\geq-1}$ is called a {\it square} if their induced subgraph in the compatibility graph (the $1$-skeleton of $\Gamma^{(m)}$) is a cycle of length 4.  The idea is to characterize the pair $\{-\alpha^1,\alpha^m\}$ among all pairs of elements in $\{ -\alpha^1, \alpha^1, \dots, \alpha^m\}$ via certain squares containing $-\beta^1$.  Note that $-\beta^1$ is compatible with elements in $\{ -\alpha^1, \alpha^1, \dots, \alpha^m\}$ (via~\eqref{rule1}), and $\{ -\alpha^1, \alpha^1, \dots, \alpha^m\}$ doesn't contain a compatible pair (via Lemma~\ref{lemm:csq_compat}).
    We check the following properties:
\begin{itemize}    
    \item $\{-\beta^1,-\alpha^1,\alpha^i\}$ can be completed to form a square if $i<m$. We take $\beta^m$ as the fourth vertex.  We have $ \beta^m \mathrel{\|} \alpha^i$ via~\eqref{rule3} ($t_\alpha t_\beta \leq c$ can be seen by taking a subword of $c=c_\bullet c_\circ$). 
    \item $\{-\beta^1,\alpha^i,\alpha^j\}$ can be completed to form a square if $1\leq i<j\leq m$.  Define $\gamma = t_\alpha(\beta)$, so that $ t_\gamma t_\alpha = t_\alpha t_\beta \leq c$. If $i>1$, this shows that we can take $\gamma^1$ as the fourth vertex of the square (via~\eqref{rule3}).  If $i=1$, note that we also have $\langle \gamma | \alpha \rangle = -\langle \beta | \alpha \rangle > 0$ (we have $\langle \beta | \alpha \rangle <0$ since $\alpha$ and $\beta$ are neighbors in the Coxeter graph).  Again, it follows that $\gamma^1$ can be chosen as the fourth vertex of a square (via~\eqref{rule2}). 
    \item $\{-\beta^1,-\alpha^1,\alpha^m\}$ cannot be completed to form a square.  Assume otherwise, and let $\gamma^i$ be the fourth vertex.
    \begin{itemize}
        \item From $-\alpha^1 \mathrel{\|} \gamma^i$, we get $t_\alpha t_\gamma \leq c$. From $ \alpha^m \mathrel{\|} \gamma^i$, we get $t_\gamma t_\alpha \leq c$.  So $t_\alpha t_\gamma = t_\gamma t_\alpha$ (this follows from the fact that elements below $c$ in the absolute order is a lattice, because $t_\alpha t_\gamma$ and $t_\gamma t_\alpha$ are rank 2 elements that covers the rank 1 elements $t_\alpha$ and $t_\gamma$).  Thus, we get $\langle \alpha |\gamma \rangle = 0$.  
        \item Write $\gamma = \sum_{\rho \in \Delta} x_\rho \cdot \rho$ with real coefficients $x_\rho \geq 0$. Since $\gamma^i$ is the fourth side of the square, we have $-\beta^1 \centernot{\|} \gamma^i$, so $x_\beta > 0$.  Moreover $-\alpha^1 \mathrel{\|} \gamma^i$ by definition of the square, so $c_\alpha = 0$.  Now, $\langle \alpha, \gamma \rangle = \sum_{\rho \in \Delta} x_\rho \langle \alpha |\rho \rangle = x_\beta \langle \alpha |\beta \rangle < 0$ (since $\alpha$ and $\beta$ are neighbors in the Coxeter graph), so that $\langle \alpha |\gamma\rangle < 0$.
    \end{itemize}
    This gives a contradiction and completes this point.
\end{itemize}

Since squares are preserved by the automorphism $\FF$, the triple $\{-\beta^1, -\alpha^1, \alpha^m\}$ is preserved as the unique one that cannot be completed to form a square among the triples considered above.  The result follows.    
\end{proof}

\begin{proof}[Proof of Proposition~\ref{prop:stabvertex}]
    The decomposition of $W_\beta$ in irreducible components is written
    \[
        W_\beta \simeq W_1 \times \dots \times W_k. 
    \]
    Up to reindexing, assume that $W_1$ is the factor that contains the reflection $t_\alpha$.  By Proposition~\ref{prop:monomial}, the restriction of $\FF$ on $\Gamma^{(m)}(W_\beta)$ permutes the irreducible factor of type $A_1$ in $W_\beta$.  These irreducible factor of type $A_1$ correspond to leaves of the Coxeter graph that are neighbors of $\beta$.  Clearly, any permutation of these leaves can be realized by an automorphism of the Coxeter graph that stabilizes $\beta$.  So, there is a diagram automorphism $\DD$ such that $\FF\DD$ stabilizes $\Gamma^{(m)}(W_1) = \{ -\alpha^1 \alpha^1 , \dots , \alpha^m \}$ (setwise) and $-\beta^1$.  
    
    By Lemma~\ref{lemm:squares}, $\FF\DD$ stabilizes $\{-\alpha^1,\alpha^m\}$.  If $\FF\DD$ stabilizes $-\alpha^1$ and $-\beta^1$, it follows from Lemma~\ref{prop:stab2vert} that it is a diagram automorphism.  Otherwise, $\SS\FF\DD$ stabilizes $-\alpha^1$ and $-\beta^1$, and it follows from Lemma~\ref{prop:stab2vert} that $\SS\FF\DD$ is a diagram automorphism.  This completes the proof. 
\end{proof}

\section{Automorphism groups}

\label{sec:autogroup}

\begin{lemm} \label{lemm:prod}
    We have $\aut(\Gamma^{(m)}) = \dih \cdot \diag$.
\end{lemm}

\begin{proof}
    Let $\FF \in \aut(\Gamma^{(m)})$, and $\alpha,\beta \in \Delta$ as in the previous sections.  There is an integer $i$ such that $\RR^i\FF(-\beta^1 ) \in - \Delta^1$.  
    
    If $\RR^i\FF(-\beta^1) = - \beta^1$, by Proposition~\ref{prop:stabvertex} we obtain $\RR^i\FF = \SS \DD $ or $\RR^i\FF = \DD $, where $\DD \in \diag$.  It follows $\FF \in \dih \cdot \diag$. 

    In the general case, let $\gamma \in \Delta$ be such that  $\beta \neq \gamma$ and $\RR^i\FF(-\beta^1) = -\gamma^1$.  The links at $-\beta^1$ and $-\gamma^1$ are isomorphic since $\RR^i\FF$ is an automorphism of $\Gamma^{(m)}$.  By Theorem~\ref{theo:nonexcep}, it follows that the Coxeter graphs of $W_{\beta}$ and $W_{\gamma}$ are isomorphic.  So, there is a diagram automorphism $\DD$ such that $\DD(-\gamma^1) = -\beta^1$ (this is easily checked at the level of Coxeter graphs).  So $\GG = \RR^i \FF \DD$ is such that $\GG(-\gamma^1) = -\gamma^1$.  If $\gamma \in \Delta_\circ$, we obtain $\GG \in \dih \cdot \diag$ by Proposition~\ref{prop:stabvertex}, and $\FF \in \dih \cdot \diag$ follows.  Otherwise, by black/white symmetry there is an analog statement (with $\TT$ in place of $\SS$), and we get $\GG \in \dih \cdot \diag$ again.
\end{proof}

\begin{lemm} \label{lemm:inter}
    We have $\dih \mathrel{\cap} \diag = \langle \CC \rangle$.
\end{lemm}

\begin{proof}   
    Note that $\CC \in \dih \mathrel{\cap} \diag$, by Proposition~\ref{prop:c_in_dih}.  It remains to check that there is no other diagram automorphism in $\dih$. 
    
    Suppose that $\RR^i \in \diag$ for some integer $i$.  By Proposition~\ref{lemm:equirepartis}, for any $\rho\in\Delta$ we have $\RR^i(-\rho^1) = -\rho^1 $ or $\RR^i(-\rho^1) = w_0(\rho)^1 $.   It follows that $\RR^i$ is either $\II$ or $\CC$.  It remains only to show that $ \RR^i \SS$ cannot be a diagram automorphism other than $\II$ or $\CC$.

    Let us first consider the case where $h$ is odd, {\it i.e.}, $W$ has type $A_{2n}$.  (It is treated separately because this is the only case where $\CC = \RR^i \SS$ for some integer $i$.)  Since $\diag = \{\II,\CC\}$, there is nothing to prove.  We assume that $h$ is even in the rest of this proof. 

    It remains to show $ \RR^i \SS \notin \diag $.  Assume otherwise and let $\rho \in \Delta_\circ$.  We get $\RR^i (  -\rho^1 ) = \RR^i \SS (  -\rho^1 ) \in -\Delta^1$.  Since $h$ is even, by Lemma~\ref{lemm:Rpow} it follows that $ i \equiv 0 \mod \frac{mh+2}{2}$.  But we then have $\RR^i = \II$ or $\RR^i = \CC$, {\it i.e.}, $\RR^i \in \diag$.  This is not possible since $\SS \notin \diag$.  
\end{proof}

\begin{lemm} \label{lemm:dist}
    We have $\dih \triangleleft \aut(\Gamma^{(m)})$.
\end{lemm}

\begin{proof}
    If $|{\diag}| = 2$, we have $[ \,\aut(\Gamma^{(m)}) : \dih \,] \leq 2$ by the previous proposition and the result follows.  This covers all cases except $D_4$. 

    If $\diag$ contains only even diagram automorphisms, it commutes with $\dih$ and the result follows from the previous lemma.  This cover all cases except $F_4$, $A_n$ ($n$ even), $I_2(k)$.
\end{proof}

\begin{theo}
    We have: 
    \[
        \aut(\Gamma^{(m)})
        =   
        \dih \rtimes (\diag/\langle \CC\rangle).
    \]
\end{theo}

\begin{proof}
    First note that $\langle \CC\rangle \triangleleft \diag$, because its index is at most 2 (in all types except $D_4$) or because it is the trivial subgroup (to include $D_4$).  
    
    If $\langle\CC\rangle$ is the trivial subgroup, the result follows from the previous lemmas.  If $\diag = \langle\CC\rangle$, we get $\aut(\Gamma^{(m)}) = \dih$ and the result is clear.  All cases are covered.
\end{proof}

\begin{coro}
    Let $\omega = |{\diag}|$. We have:
    \[
        \big|\aut(\Gamma^{(m)})\big| = (mh+2)\omega. 
    \]
\end{coro}

\begin{proof}
    By the previous theorem, we have
    \[
         \big|\aut(\Gamma^{(m)})\big|
         =
         \frac{|{\dih}|}{|\CC|}\omega.
    \]
    Moreover, $|{\dih}| = 2|\RR|$.  We have $|{\RR}| = \frac{mh+2}{2}|{w_0}|$ by~\eqref{eq:orderR}, and clearly $|{w_0}| = |{\CC}|$.  Putting this together completes the proof.    
\end{proof}

The final step is the following generalization of Proposition~\ref{prop:stabvertex}, without restriction on the chosen vertex of the Coxeter diagram.  Recall that it was used in the proof of Proposition~\ref{prop:stab2vert} (where it was assumed to hold via an induction hypothesis).

\begin{prop} \label{prop:stabvertex2}
    Let $\rho \in \Delta_\circ$ and $\FF \in \aut(\Gamma^{(m)})$ be such that $\FF(-\rho^1) = -\rho^1$.  Then we have either $\FF = \DD$ or $\FF = \SS \DD$, where $\DD$ is an even diagram automorphism such that $\DD(-\rho^1) = -\rho^1$.  (And by black/white symmetry, there is a similar statement with $\Delta_\bullet$ and $\TT$ in place of $\Delta_\circ$ and $\SS$.)
\end{prop} 

\begin{proof}
    We use the previous theorem and the orbit-stabilizer theorem.  This suffices to show that the stabilizer contains only the listed elements.  In most cases, the stabilizer has order 2 and it follows that it is $\{ \II, \SS\}$.  We only give details about the other cases:
    \begin{itemize}
        \item If $\omega = 2$ and the orbit of $-\rho^1$ has cardinality $\frac{mh+2}{2}$: it means that the stabilizer has order 4, moreover $-\rho^1$ is fixed by the nontrivial element $\DD \in \diag$.  It follows that the stabilizer is $\{ \II, \SS, \DD, \SS\DD \}$ (we have $\SS\DD = \DD\SS$ since $\DD$ is even).
        \item In the the case of $D_4$, $\omega=6$ and the orbit of $-\rho^1$ has cardinality $\frac{mh+2}{2}$.  The stabilizer has cardinality $4$ if $\rho$ is a leaf of the Coxeter diagram and $12$ if it is the central vertex.  There are 2 elements of $\diag$ stabilizing $-\rho^1$ in the first case, and 6 in the second case.  This completes the proof.
    \end{itemize}
All cases have been covered.
\end{proof}





    


\section{Final remarks}

\label{sec:final}

Understanding the automorphism group of a combinatorial object is certainly interesting on its own, but let us mention some problems that motivate the present work.  

{\it Cluster parking functions} were recently introduced in a joint work with Douvropoulos~\cite{douvropoulosjosuatverges}.  They can be seen combinatorially as faces of $\Gamma^{(m)}$ with additional structure (essentially a labeling by a coset in $W$).  It is conjectured that symmetries of $\Gamma^{(m)}$ have a counterpart in cluster parking functions (for example, see~\cite[Open problem~9.3]{douvropoulosjosuatverges}).  This is completely apparent in the combinatorial models in types $A$ and $B$, and a general Coxeter-theoretic result in this direction would be very interesting. 

As briefly mentioned in the introduction, the generalized cluster complex has a natural representation-theoretic interpretation (via an orbit category in the derived category of certain algebras, see~\cite{buan} for a survey).  In this context, $\RR$ naturally identifies with a {\it shift functor}.  It would be very interesting to also give a representation-theoretic interpretation of the involutive automorphisms in this context.

\end{document}